\documentclass[11pt]{article}

\usepackage[T1]{fontenc}
\usepackage{lmodern}
\usepackage{amsmath,amssymb,amsfonts,amsthm}
\usepackage{mathtools}
\usepackage{graphicx}
\usepackage{enumerate}
\usepackage{hyperref}
\usepackage[a4paper,left=3cm,right=3cm,top=3cm,bottom=3cm]{geometry}

\hypersetup{
  colorlinks=true,
  linkcolor=black,
  citecolor=black,
  urlcolor=black
}

\newtheorem{theorem}{Theorem}[section]
\newtheorem{lemma}[theorem]{Lemma}
\newtheorem{corollary}[theorem]{Corollary}

\newtheorem{question}[theorem]{Question}

\theoremstyle{definition}
\newtheorem{definition}[theorem]{Definition}

\theoremstyle{remark}
\newtheorem{remark}[theorem]{Remark}

\def\1{1\!\hspace{-.08cm}1}

\title{Extremal problems on the $p$-Seidel energy of graphs}

\author{Alexander Guterman\thanks{Department of Mathematics, Bar-Ilan University, Ramat-Gan, Israel (alexander.guterman@biu.ac.il)}
\and
Shib Sankar Saha\thanks{Corresponding author: Department of Mathematics, Bar-Ilan University, Ramat-Gan, Israel (shibkol2019@gmail.com/sahashi@biu.ac.il)}
}

\begin{document}
\maketitle

\begin{abstract}
Let $G$ be a graph with vertex set $\{v_1,\ldots,v_n\}$. The \emph{Seidel matrix} $S(G)$ is the $n\times n$ symmetric matrix with
zero diagonal whose $(i,j)$-entry, for $i\ne j$, is $-1$ if $v_i$ and $v_j$ are adjacent and $1$ otherwise. If $\lambda_1,\ldots,\lambda_n$ are the Seidel eigenvalues of $G$, then the $p$-Seidel energy of $G$ is defined by
\[
\mathcal{E}_p(S(G))
=
\sum_{i=1}^n |\lambda_i|^p,
\qquad p>0.
\]
For $p>2$, we establish a universal lower bound for the $p$-Seidel energy of graphs of order $n$ and characterize the equality case in terms of symmetric conference Seidel matrices. Consequently, whenever a symmetric conference matrix of order $n$ exists, this bound determines the minimum $p$-Seidel energy and all graphs attaining it. For $0<p<2$, we obtain the corresponding upper bound and equality characterization; hence, whenever a symmetric conference matrix of order $n$ exists, the maximum $p$-Seidel energy is determined. In addition, for $p>2$, $n\equiv1\pmod4$, and $r=\frac{n-1}{2}$, we establish a lower bound for the $p$-Seidel energy among all $r$-regular graphs of order $n$ and characterize its equality case. Consequently, whenever a conference graph of order $n$ exists, this bound determines the minimum $p$-Seidel energy and all graphs attaining it. In particular, when $n$ is a prime power with $n\equiv1\pmod4$, the Paley graph $P(n)$ is a conference graph and hence attains this minimum. For $p\ge3$ and $n=2r$, we determine the maximum $p$-Seidel energy among all $r$-regular graphs of order $n$ and characterize the equality case. We also determine the maximum $p$-Seidel energy among $r$-regular graphs with $n=2r+2$ for $p\geq3$ and the maximum is attained uniquely by the disjoint union of two complete graphs of order $r+1$. Finally, we pose several open problems concerning the $p$-Seidel energy for different values of $p$.
\end{abstract}



\sloppy
\section{Introduction}
Let $G=(V,~E)$ be a simple graph with $n$ vertices ($n$ is called the order of $G$) and $m$ edges ($m$ is called the size of $G$) having vertex set $V=\{v_1,v_2,\ldots,v_n\}$ and edge set $E=\{e_1,e_2,\ldots,e_m\}$. The adjacency matrix $A(G)=[A_{ij}]_{n\times n}$ of $G$ is a $(0, 1)$-square matrix of order $n$ whose $ij$-entry is equal to 1 if $v_i$ is adjacent to $v_j$ and equal to 0 otherwise. The adjacency matrix of a graph is a real symmetric matrix. The \emph{degree} of a vertex $i\in V$ in a graph $G$ is the number of vertices adjacent to~$i$. The \emph{complement} of a graph $G$ is the graph on the same set of vertices in which two distinct vertices are adjacent if and only if they are not adjacent in $G$, and is denoted by $\overline{G}$. The \emph{complete graph} of order $n$ is the graph with $n$ vertices in which every pair of distinct vertices is adjacent, and is denoted by~$K_n$. A graph $G$ is said to be \emph{$r$-regular} if each vertex of $G$ has degree~$r$. A graph $G$ is \emph{complete bipartite} if $V=X\sqcup Y$ with $X,Y\neq\varnothing$ and $v_iv_j\in E$ for all $v_i\in X$, $v_j\in Y$, and no edges join two vertices within $X$ or within $Y$. When $|X|=|Y|$, the graph is called \emph{balanced}. If $|X|=a$ and $|Y|=b$, the graph is denoted by $K_{a,b}$. Throughout this article, all graphs are non-empty, undirected, simple, finite, and have at least one edge. Whenever connectedness is required, it will be stated explicitly.

Let $M$ be a Hermitian matrix of order $n$. Then $M$ has $n$ real eigenvalues, denoted by \(\alpha_1(M),\alpha_2(M),\ldots\), \(\alpha_n(M)\). The \emph{$p$-energy} of $M$, denoted by $\mathcal{E}_p(M)$, is defined to be the sum of the $p$th powers of the absolute values of the eigenvalues of $M$ i.e.,

\begin{center}
$\mathcal{E}_p(M)=\sum\limits_{i=1}^n|\alpha_i|^p,~~p>0$.\end{center}
The well-known concept of the \emph{energy} of a graph $G$, denoted by $\mathcal{E}(G)$, is $\mathcal{E}_1(A)$, where $A=A(G)$ is the adjacency matrix of $G$. In 2025, Akbari, Kumar, Mohar, and Pragada studied the positive and negative square energies of a graph \cite{S2025}. The {\em Seidel matrix} $S(G)=[S(G)_{ij}]_{n\times n}$ of a graph $G=(V,~E)$ is a matrix of order $n$ defined by: 
\[
S(G)_{ij} =
\begin{cases}
0, & v_i=v_j,\\
-1, & v_i\neq v_j \text{ and }v_iv_j \in E,\\
+1, & v_i\neq v_j \text{ and }v_iv_j \notin E.
\end{cases}
\]
In 1966, van Lint and Seidel introduced the concept of the Seidel matrix for graphs and studied its connection with equiangular lines in~\cite{S1996}. The Seidel matrix can also be written as $S(G)=J-I-2A(G)$, where $A(G)$, $J$, and $I$ are the adjacency matrix of $G$, the matrix with all entries $1$, and the identity matrix, respectively. For more details, see \cite{S2020, WH2012, WH2010}.

Let $\lambda_1,\lambda_2,\ldots,\lambda_n$ be the Seidel eigenvalues of $G$. The notation $\mathbf{1}$ is the vector $(1,1,\ldots,1)^\top\in\mathbb{R}^n$. The notations $\operatorname{tr}(M)$ and $\operatorname{Spec}(M)$ are the trace and collection of eigenvalues of the matrix $M$, respectively. Suppose that $G=(V,~E)$ has $q$ distinct Seidel eigenvalues. Then the Seidel spectrum of $G$ can be written as
$$\operatorname{Spec}(S(G))=
\big\{
\underbrace{\lambda_1, \ldots, \lambda_1}_{m_1~\text{times}}, 
\underbrace{\lambda_2, \ldots, \lambda_2}_{m_2~\text{times}}, 
\ldots,
\underbrace{\lambda_q, \ldots, \lambda_q}_{m_q~\text{times}} 
\big\},
$$
where $m_i$ is the algebraic multiplicity of $\lambda_i$, for $1\leq i\leq q$ with $\sum\limits_{i=1}^q m_i=n$.

In 2020, Akbari, Einollahzadeh, Karkhaneei, and Nematollahi defined the \emph{$p$-Seidel energy} of $G$ by
\begin{center}
$\mathcal{E}_p(S(G))=\sum\limits_{i=1}^n|\lambda_i|^p,~~p>0$,    
\end{center}
where $|\lambda_i|$ denotes the absolute value of $\lambda_i$ \cite[Page 2]{S2020}. For $p=1$, the $1$-Seidel energy is the usual Seidel energy of \(G\), denoted by $\mathcal{E}(S(G))$.

\begin{definition}
Let $G=(V,E)$ be a graph, and let \(V=V_1\sqcup V_2
\) be a partition of its vertex set. The \emph{Seidel switching} of $G$ with respect to $V_1$ (equivalently, with respect to the partition
$(V_1,V_2)$) is the graph $G'$ obtained from $G$ as follows: the adjacency relation between $V_1$ and $V_2$ is reversed, while the adjacency relations within $V_1$ and within $V_2$ remain unchanged. More precisely, for $u\in V_1$ and $v\in V_2$,
\[
uv\in E(G')
\quad\Longleftrightarrow\quad
uv\notin E(G),
\]
whereas for two vertices belonging to the same part,
\[
uv\in E(G')
\quad\Longleftrightarrow\quad
uv\in E(G).
\]
\end{definition}
Let \(D=\operatorname{diag}(d_1,\ldots,d_n)\), where \(d_i=
\begin{cases}
1, & v_i\in V_1,\\
-1, & v_i\in V_2.
\end{cases}\)
If $G'$ is obtained from $G$ by the above Seidel switching, then
\begin{equation}\label{eq:Seidel-switching}
S(G')=DS(G)D.
\end{equation}
Indeed, if two vertices belong to the same part, then $d_i d_j=1$, so the corresponding entry of the Seidel matrix is unchanged. If they belong to different parts, then $d_i d_j=-1$, so the corresponding Seidel entry changes sign, which is exactly the effect of reversing adjacency between the two parts.

Since $D^T=D$ and $D^2=I$, the matrices $S(G)$ and $S(G')$ are orthogonally similar. Therefore,
\[
\operatorname{Spec}(S(G'))
=
\operatorname{Spec}(S(G)).
\]
Consequently, Seidel switching preserves the $p$-Seidel energy for every $p>0$:
\[
\mathcal{E}_p(S(G'))
=
\mathcal{E}_p(S(G)).
\]
Two graphs are called \emph{$S$-equivalent} if one can be obtained from the other by a Seidel switching. They are called \emph{$SC$-equivalent} if one graph is $S$-equivalent to either the other graph or its complement \cite{WH2012}.

In this article, we study extremal problems for the \(p\)-Seidel energy over graphs and regular graphs of fixed order. The paper is organized as follows.

In Section~\ref{secc2}, for $p>2$, we establish a universal lower bound for the $p$-Seidel energy of graphs of order $n$ and characterize its equality
case in terms of symmetric conference Seidel matrices. Consequently, whenever a symmetric conference matrix of order $n$ exists, this bound determines the exact minimum $p$-Seidel energy and all graphs attaining it. For $0<p<2$, we obtain the corresponding upper bound and equality
characterization. Hence, whenever a symmetric conference matrix of order $n$ exists, the exact maximum $p$-Seidel energy is determined. When no such matrix exists, the corresponding inequality is strict. In Section~\ref{secc3}, for $p>2$, $n\equiv1\pmod4$, and $r=\frac{n-1}{2}$, we establish a lower bound for the $p$-Seidel energy among all $r$-regular graphs of order $n$ and characterize its equality
case. Consequently, whenever a conference graph of order $n$ exists, this bound determines the minimum $p$-Seidel energy and all graphs attaining it. In particular, when $n$ is a prime power with $n\equiv1\pmod4$, the Paley graph $P(n)$ is a conference graph and hence attains this minimum. In Section~\ref{secc4}, we characterize the graph that achieves the maximum $p$-Seidel energy among all $r$-regular graphs with fixed order $n=2r$, where $p\geq3$. We also determine the maximum $p$-Seidel energy among $r$-regular
graphs of order $2r+2$ for $p\geq3$. In this case, the maximum is attained uniquely by the disjoint union of two complete graphs of order $r+1$. In Section~\ref{seccc5}, we state open problems concerning the $p$-Seidel energy for different values of $p$.

\section{A lower bound for the $p$-Seidel energy, $p>2$}\label{secc2}

Haemers studied the concept of Seidel energy of a graph, and he proposed the conjecture, which is later became known as {\em Haemers' Conjecture}, i.e., the complete graph has minimum Seidel energy among all graphs with $n$ vertices \cite{WH2012}. Akbari, Einollahzadeh, Karkhaneei, and Nematollahi proved the conjecture and showed that equality is characterized up to $SC$-equivalence \cite{S2020}. Also in 2024, Einollahzadeh and Nematollahi gave a short proof of {\em Haemers' Conjecture} on the Seidel energy of graphs \cite{ME2024}. In particular, the complete graph is not the unique equality graph up to isomorphism.

\begin{theorem}\rm\cite[Theorem 2]{S2020}\label{Th.2.1}
Let $G$ be a graph with $n$ vertices. Then
\begin{center}
$\mathcal{E}(S(G))\geq\mathcal{E}(K_n),$    
\end{center}
and equality holds if and only if $G$ is $SC$-equivalent to $K_n$.
\end{theorem}
Also, in \cite{S2020}, the following strict lower bound of $p$-Seidel energy of graphs for any $p\in(0,2)$ is established.

\begin{theorem}\rm\cite[Theorem 1]{S2020}\label{Th.2.2}
Let $G$ be a graph with $n$ vertices. Then, for every real number $p\in(0,2)$
\begin{center}
$\mathcal{E}_p(S(G))>(n-1)^p+(n-2).$    
\end{center}
\end{theorem}

\begin{remark}\label{Remarkkk}
Since $S(G)$ is a real symmetric matrix, its squared Frobenius norm is
\[
\|S(G)\|_F^2
=
\operatorname{tr}\bigl(S(G)^\top S(G)\bigr)
=
\operatorname{tr}\bigl(S(G)^2\bigr).
\]
The diagonal entries of $S(G)$ are zero and all its off-diagonal entries are equal to $\pm1$. Therefore,
\[
\|S(G)\|_F^2
=
\sum_{i,j=1}^n |S(G)_{ij}|^2
=
n(n-1).
\]
On the other hand, if $\lambda_1,\ldots,\lambda_n$ are the Seidel eigenvalues of $G$, then
\[
\operatorname{tr}\bigl(S(G)^2\bigr)
=
\sum_{i=1}^n\lambda_i^2.
\]
Consequently,
\begin{equation}\label{ERT}
\sum_{i=1}^n\lambda_i^2
=
\|S(G)\|_F^2
=
n(n-1).
\end{equation}
Thus, the Euclidean norm of the Seidel spectrum is fixed for every graph of order $n$. In particular,
\[
\mathcal E_2(S(G))
=
\sum_{i=1}^n|\lambda_i|^2
=
n(n-1),
\]
so the $2$-Seidel energy is independent of the structure of $G$. Therefore, the extremal $p$-Seidel energy problem is nontrivial only for $0<p<2$ or $p>2$.
\end{remark}
The discussion in Remark \ref{Remarkkk} leads to the following question: 

\begin{question}\label{qe.3.1}
What is the minimum $p$-Seidel energy among all graphs of order $n$, for $p>2$, and which graphs attain it?
\end{question}
To answer Question \ref{qe.3.1}, we first recall the definition of a conference matrix.

\begin{definition}\rm\cite[Page 658]{WH2012}\label{D.2.5}
A \emph{conference matrix} is a square matrix $S$ of order $n$ with zero diagonal and $\pm1$ off-diagonal entries such that
\[
S S^{\top} = (n-1)I.
\]
We say that a graph $G$ has a \emph{symmetric conference Seidel matrix} if its Seidel matrix $S(G)$ is a symmetric conference matrix.
\end{definition}

\begin{remark}
Conference matrices were introduced by Belevitch in connection with ideal telephone conference networks \cite{Belevitch1950}. Throughout this article, whenever a symmetric conference Seidel matrix of order $n$ is considered, we assume that $n$ is an order for which such a matrix exists.  The subsequent combinatorial formulation linked these matrices with conference graphs \cite{VanLintWilson2001}.
\end{remark}
The following lemma is known as the classical power mean inequality (also called the generalized mean inequality).

\begin{lemma}\rm\cite[Theorem 1]{PSB2013}\label{lemma1}
Let $a_1, a_2, \ldots, a_n$ be nonnegative real numbers and let $r>s>0$. Then
\[
\left( \frac{1}{n} \sum_{i=1}^{n} a_i^{r} \right)^{\frac{1}{r}}
\ge 
\left( \frac{1}{n} \sum_{i=1}^{n} a_i^{s} \right)^{\frac{1}{s}},
\]
with equality if and only if $a_1 = a_2 = \cdots = a_n$.
\end{lemma}
In the following theorem, we establish a universal lower bound for the \(p\)-Seidel energy and characterize its equality case. Consequently, for every order \(n\) admitting a symmetric conference matrix, the theorem determines the minimum \(p\)-Seidel energy, which also answers Question \ref{qe.3.1} for such $n$.

\begin{theorem}\label{mm1}
Let $G$ be a graph on $n \ge 2$ vertices. Then, for every $p>2$,
\[
\mathcal{E}_p(S(G)) \;\ge\; n\,(n-1)^{\frac{p}{2}},
\]
with equality if and only if \(S(G)^2=(n-1)I\), that is, if and only if $S(G)$ is a symmetric conference matrix.
\end{theorem}

\begin{proof} Let $\lambda_i$ be the Seidel eigenvalues of $G$, for $i=1,2,\ldots,n$.

Substituting $a_i=|\lambda_i|$ into Lemma \ref{lemma1} and using (\ref{ERT}) we obtain
\[
\mathcal{E}_p(S(G))
= \sum_{i=1}^n |\lambda_i|^p
= \sum_{i=1}^n a_i^p
\;\ge\;
n\left(\frac{1}{n}\sum_{i=1}^n a_i^2\right)^{\frac{p}{2}}
=
n\left(\frac{1}{n}\sum_{i=1}^n \lambda_i^2\right)^{\frac{p}{2}}
=
n\,(n-1)^{\frac{p}{2}},
\]
equality holds if and only if \(
|\lambda_1|=|\lambda_2|=\cdots=|\lambda_n|\).

Suppose that equality holds, and let \(|\lambda_1|=\cdots=|\lambda_n|=a\). 

By \eqref{ERT}, \(na^2
=
\sum_{i=1}^n \lambda_i^2
=
n(n-1)\), and hence \(a^2=n-1\). Therefore, \(\lambda_i^2=n-1\), \(i=1,\ldots,n\).

Since $S(G)$ is real symmetric, there exists an orthogonal matrix $Q$ such that
\[
S(G)=Q\Lambda Q^{\top},
\]
where \(\Lambda=\operatorname{diag}(\lambda_1,\ldots,\lambda_n)\). 

Consequently,
\[
S(G)^2
=
Q\Lambda^2Q^{\top}
=
Q\bigl((n-1)I\bigr)Q^{\top}
=
(n-1)I.
\]
Since $S(G)$ is symmetric, has zero diagonal, and has $\pm1$
off-diagonal entries, it follows that $S(G)$ is a symmetric conference
matrix.

Conversely, suppose that \(S(G)^2=(n-1)I\). If $\lambda_i$ is an eigenvalue of $S(G)$, then \(
\lambda_i^2=n-1\), and therefore
\[
|\lambda_i|=\sqrt{n-1}
\qquad\text{for all }i=1,\ldots,n.
\]
Hence
\[
\mathcal{E}_p(S(G))
=
\sum_{i=1}^n|\lambda_i|^p
=
n(n-1)^{\frac{p}{2}}.
\]
Thus equality holds if and only if $S(G)$ is a symmetric conference matrix.

This completes the proof.
\end{proof}

\begin{remark}\label{rem:conference-existence}
The lower bound in Theorem~\ref{mm1} is attained only for those orders $n$ for which a symmetric conference matrix of order $n$ exists. Indeed, if no such matrix exists, then the equality characterization in Theorem~\ref{mm1} implies that
\[
\mathcal{E}_p(S(G))
>
n(n-1)^{\frac{p}{2}}
\]
for every graph $G$ of order $n$.

There are well-known necessary arithmetic conditions for the existence of a symmetric conference matrix. In particular, for $n>2$, if a symmetric conference matrix of order $n$ exists, then
\[
n\equiv2\pmod4
\]
and
\[
n-1=a^2+b^2
\]
for some integers $a$ and $b$; see, for example,
\cite[Lemma~10.4.1 and Theorem~10.4.2, p.~158]{AEB2011}. Thus, if either $n\not\equiv2\pmod4$ or $n-1$ is not representable as a sum of two integer squares, then no symmetric conference matrix of order $n$ exists, and hence the inequality in Theorem~\ref{mm1} is strict.
\end{remark}

\begin{remark}
For $0<p<2$, let $\lambda_1,\ldots,\lambda_n$ be the Seidel
eigenvalues of a graph $G$ of order $n$. Applying Lemma~\ref{lemma1} to $a_i=|\lambda_i|$ and using \eqref{ERT}, we obtain
\[
\mathcal{E}_p(S(G))
= \sum_{i=1}^n |\lambda_i|^p
= \sum_{i=1}^n a_i^p
\;\le\;
n\left(\frac{1}{n}\sum_{i=1}^n a_i^2\right)^{\frac{p}{2}}
=
n\left(\frac{1}{n}\sum_{i=1}^n \lambda_i^2\right)^{\frac{p}{2}}
=
n\,(n-1)^{\frac{p}{2}},
\]
Equality holds if and only if \(|\lambda_1|=\cdots=|\lambda_n|=\sqrt{n-1}\). Since $S(G)$ is real symmetric, this is equivalent to \(S(G)^2=(n-1)I\), that is, $S(G)$ is a symmetric conference matrix.

Consequently, if a symmetric conference matrix of order $n$ exists, then
\[
\max_{|V(G)|=n}\mathcal{E}_p(S(G))
=
n(n-1)^{\frac{p}{2}},
\]
and equality holds precisely for those graphs whose Seidel matrices are symmetric conference matrices.

If no symmetric conference matrix of order $n$ exists, then
\[
\mathcal{E}_p(S(G))
<
n(n-1)^{\frac{p}{2}}
\]
for every graph $G$ of order $n$. In this case, the above inequality does not determine the exact maximum.
\end{remark}

\section{A lower bound for the $p$-Seidel energy of $r$-regular graphs, $p>2$}\label{secc3}

\begin{lemma}\label{lem:PM}
Let $G$ be a graph with $n$ vertices. Let $p>2$ and $\lambda_1,\lambda_2,\lambda_3,\dots,\lambda_n$ be the Seidel eigenvalues of $G$. Then
\[
\sum_{i=2}^{n}|\lambda_i|^p
\;\ge\;
(n-1)^{1-\frac p2}\!
\big(n(n-1)-\lambda_1^2\big)^{\frac p2}
\]
with equality if and only if
$|\lambda_2|=\cdots=|\lambda_n|$.
\end{lemma}

\begin{proof} Let $t=n-1$ and $z_i=|\lambda_{i+1}|$ for $i=1,\dots,t$. Then we have
\[
\sum_{i=2}^n |\lambda_i|^p=\sum_{i=1}^t z_i^p
\quad\text{and}\quad
\sum_{i=2}^n \lambda_i^2=\sum_{i=1}^t z_i^2.
\]
For $p>2$ and using Lemma \ref{lemma1} we have
\[
\left(\frac1t\sum_{i=1}^t z_i^p\right)^{\!\frac{1}{p}}
\;\ge\;
\left(\frac1t\sum_{i=1}^t z_i^2\right)^{\!\frac{1}{2}},
\]
equality holds if and only if all $z_{i}$ are equal for $i=1,\dots,t$.

\[
\hspace{1cm}\implies\sum_{i=1}^t z_i^p
\;\ge\;
t^{\,1-\frac p2}\left(\sum_{i=1}^t z_i^2\right)^{\!\frac p2},
\]
equality holds if and only if all $z_{i}$ are equal for $i=1,\dots,t$.

\[
\hspace{1cm}\implies\sum_{i=2}^{n}|\lambda_i|^p
\;\ge\;
(n-1)^{1-\frac p2}\!
\left(\sum_{i=2}^{n}\lambda_i^2\right)^{\frac p2},
\]
equality holds if and only if all $z_{i}=|\lambda_{i+1}|$ are equal for $i=1,\dots,(n-1)$.

Now using (\ref{ERT}) we have $\sum\limits_{i=2}^n \lambda_i^2=n(n-1)-\lambda_1^2$, which gives 
\[\sum_{i=2}^{n}|\lambda_i|^p
\;\ge\;
(n-1)^{1-\frac p2}\!
\big(n(n-1)-\lambda_1^2\big)^{\frac p2},
\]
equality holds if and only if all $|\lambda_{i+1}|$ are equal for $i=1,\dots,(n-1)$, i.e., when $|\lambda_2|=|\lambda_3|=\cdots=|\lambda_n|$.

This completes the proof.
\end{proof}

\begin{lemma}\rm\cite[Page 221]{AEB2011}\label{lem:t}
Let $G$ be an $r$-regular graph on $n$ vertices with adjacency eigenvalues \(\theta_1 = r,\theta_2,\dots,\theta_n\), where $\theta_1=r$ corresponds to the eigenvector $\mathbf{1}$. Then the Seidel eigenvalue corresponding to $\mathbf{1}$ is \(n-1-2r\), and the remaining Seidel eigenvalues are \(-1-2\theta_i\), \(i=2,\ldots,n\).
\end{lemma}

\begin{theorem}\label{LB}
Let $G$ be an $r$-regular graph on $n$ vertices and $\lambda_1,\lambda_2,\dots,\lambda_n$ be the Seidel eigenvalues of $G$ such that $\lambda_1$ is the eigenvalue corresponding to $\mathbf{1}$. Then for every $p>2$,
\begin{center}
$\mathcal{E}_p(S(G))\;\ge\;|n-1-2r|^p+(n-1)^{1-\frac p2}
\!\big(n(n-1)-(n-1-2r)^2\big)^{\!\frac p2},$
\end{center}
with equality if and only if $|\lambda_2|=\cdots=|\lambda_n|$.
\end{theorem}

\begin{proof} 
By Lemma \ref{lem:t}, since $\lambda_1$ is the Seidel eigenvalue corresponding to $\mathbf{1}$, we have
\[
\lambda_1=n-1-2r.
\]
The result now follows directly from Lemma \ref{lem:PM}. 
\end{proof}

\begin{definition}\rm\cite[Page 115]{AEB2011}
A graph $G$ with $n$ vertices is called a \emph{strongly regular graph} with parameters $(n,k,\lambda,\mu)$, written as $G=\operatorname{srg}(n,k,\lambda,\mu)$, where $G$ is $k$-regular, any two adjacent vertices have exactly $\lambda$ common neighbors, and any two non-adjacent vertices have exactly $\mu$ common neighbors.
\end{definition}

\begin{definition}\rm\cite[Page 115]{AEB2011}
A strongly regular graph with parameters
\[
\left(
n,\frac{n-1}{2},
\frac{n-5}{4},
\frac{n-1}{4}
\right)
\]
is called a \emph{conference graph}. Equivalently, a conference graph
$G$ of order $n$ satisfies
\[
G=
\operatorname{srg}\!\left(
n,\frac{n-1}{2},
\frac{n-5}{4},
\frac{n-1}{4}
\right).
\]
\end{definition}

\begin{remark}
The notion of a conference graph above should not be confused with a graph whose Seidel matrix is itself a symmetric conference matrix. If $G$ is a conference graph of order $n$, then
\[
S(G)\mathbf 1=0
\]
and
\[
S(G)^2=nI-J.
\]
Thus $S(G)$ is singular and is not a conference matrix.

The two notions are nevertheless closely related. After a suitable switching, a symmetric conference matrix of order $n+1$ can be written
as
\[
C=
\begin{pmatrix}
0&\mathbf1^{\top}\\
\mathbf1&S(G)
\end{pmatrix},
\]
where $G$ is a conference graph of order $n$.
\end{remark}

\begin{definition}\rm\cite[Page 115]{AEB2011}
Let $n$ be a prime power with $n\equiv1\pmod 4$, and let $\mathbb{F}_n$ be the finite field of order $n$. The \emph{Paley graph} $P(n)$ is the graph with vertex set $\mathbb{F}_n$ in which two
distinct vertices $x$ and $y$ are adjacent if and only if $x-y$ is a nonzero square in $\mathbb{F}_n$.
\end{definition}

\begin{remark}\rm\cite[Proposition 9.1.1]{AEB2011}\label{Note}
The Paley graph $P(n)$ is a conference graph. The adjacency spectrum of a conference graph $G$ is
\[
\operatorname{Spec}(A(G)) =
\Bigg\{
\frac{n-1}{2},\
\underbrace{
\frac{-1+\sqrt{n}}{2},\dots,\frac{-1+\sqrt{n}}{2}
}_{\frac{n-1}{2}~\text{times}},\
\underbrace{
\frac{-1-\sqrt{n}}{2},\dots,\frac{-1-\sqrt{n}}{2}
}_{\frac{n-1}{2}~\text{times}}
\Bigg\}.
\]
\end{remark}
In the following theorem, for $p>2$, $n\equiv1\pmod4$, and $r=\frac{n-1}{2}$, we establish a lower bound for the $p$-Seidel energy among all $r$-regular graphs of order $n$ and characterize the equality case. Whenever a conference graph of order $n$ exists, this bound gives the exact minimum. In particular, if $n$ is a prime power, the Paley graph $P(n)$ attains this minimum.

\begin{theorem}\label{conference}
Let $n\geq5$ with $n\equiv1\pmod 4$, and let \(r=\frac{n-1}{2}\). Let $G$ be an $r$-regular graph on $n$ vertices. Then, for every $p>2$,
\[
\mathcal{E}_p(S(G))
\geq
(n-1)n^{\frac p2}.
\]
Equality holds if and only if $G$ is a conference graph. In particular, whenever $n$ is a prime power, the Paley graph $P(n)$ attains equality.
\end{theorem}

\begin{proof}
Label the Seidel eigenvalues of $G$ as \(
\lambda_1,\lambda_2,\ldots,\lambda_n\) so that $\lambda_1$ is the eigenvalue corresponding to $\mathbf1$. Since $G$ is $r$-regular, Lemma~\ref{lem:t} gives
\[
\lambda_1=n-1-2r.
\]
Since $r=\frac{n-1}{2}$, we have \(\lambda_1=0\). Therefore, by Theorem~\ref{LB},
\begin{align}
\mathcal{E}_p(S(G))
&\geq
(n-1)^{1-\frac p2}
\bigl(n(n-1)\bigr)^{\frac p2}\notag\\
&=
(n-1)n^{\frac p2}.
\label{eq:LB-special}
\end{align}
Equality holds if and only if \(
|\lambda_2|
=
|\lambda_3|
=
\cdots
=
|\lambda_n|\).
Suppose first that equality holds. Then there exists $\alpha>0$ such that \(\lambda_i\in\{\alpha,-\alpha\}\), \(i=2,\ldots,n\).

Let $m_+$ and $m_-$ denote the multiplicities of $\alpha$ and $-\alpha$, respectively. Since \(
\operatorname{tr}(S(G))
=
\sum_{i=1}^n\lambda_i
=0\) and $\lambda_1=0$, we obtain
\[
\alpha(m_+-m_-)=0.
\]
Hence
\[
m_+=m_-=\frac{n-1}{2}.
\]
Furthermore, using \eqref{ERT},
\[
n(n-1)
=
\sum_{i=1}^n\lambda_i^2
=
(m_++m_-)\alpha^2
=
(n-1)\alpha^2.
\]
Thus \(\alpha^2=n\), and consequently
\[
\operatorname{Spec}(S(G))
=
\Bigg\{
0,\
\underbrace{
-\sqrt{n},\dots,-\sqrt{n}
}_{\frac{n-1}{2}~\text{times}},\
\underbrace{
\sqrt{n},\dots,\sqrt{n}
}_{\frac{n-1}{2}~\text{times}}
\Bigg\}.
\]
We now show that this equality condition implies that $G$ is a conference graph. 

Let \(r=\frac{n-1}{2}\) be the trivial adjacency eigenvalue of $G$. By Lemma~\ref{lem:t}, the remaining adjacency eigenvalues are obtained from
\[
\lambda=-1-2\theta.
\]
Hence they are
\[
\theta_+=\frac{-1+\sqrt n}{2},
\qquad
\theta_-=\frac{-1-\sqrt n}{2},
\]
each with multiplicity $\frac{n-1}{2}$.

Let $A=A(G)$. On the subspace orthogonal to $\mathbf1$, the two nontrivial eigenvalues satisfy
\[
x^2+x-\frac{n-1}{4}=0.
\]
Let \(M=A^2+A-\frac{n-1}{4}I\). As shown above, $M$ vanishes on $\mathbf1^\perp$, while \(M\mathbf1
=
\frac{n(n-1)}{4}\mathbf1\). 

Since \(\mathbb R^n
=
\operatorname{span}\{\mathbf1\}\oplus\mathbf1^\perp\) and the orthogonal projection onto $\operatorname{span}\{\mathbf1\}$ is $\frac1nJ$, it follows that
\[
M
=
\frac{n(n-1)}{4}\left(\frac1nJ\right)
=
\frac{n-1}{4}J.
\]
Hence
\[
A^2+A-\frac{n-1}{4}I
=
\frac{n-1}{4}J,
\]
or equivalently,
\begin{equation}\label{eq:conference-A2}
A^2
=
\frac{n-1}{4}I
-A
+\frac{n-1}{4}J.
\end{equation}
For distinct vertices $i$ and $j$, the entry $(A^2)_{ij}$ is the number of their common neighbors. If $i$ and $j$ are adjacent, then $A_{ij}=1$, and \eqref{eq:conference-A2} gives
\[
(A^2)_{ij}
=
-1+\frac{n-1}{4}
=
\frac{n-5}{4}.
\]
If $i$ and $j$ are nonadjacent, then $A_{ij}=0$, and hence
\[
(A^2)_{ij}
=
\frac{n-1}{4}.
\]
Thus
\[
G=
\operatorname{srg}\!\left(
n,\frac{n-1}{2},
\frac{n-5}{4},
\frac{n-1}{4}
\right),
\]
so $G$ is a conference graph.

Conversely, suppose that $G$ is a conference graph. Then
\[
G=
\operatorname{srg}\!\left(
n,\frac{n-1}{2},
\frac{n-5}{4},
\frac{n-1}{4}
\right).
\]
The adjacency spectrum of $G$ is
\[
\operatorname{Spec}(A(G)) =
\Bigg\{
\frac{n-1}{2},\
\underbrace{
\frac{-1+\sqrt{n}}{2},\dots,\frac{-1+\sqrt{n}}{2}
}_{\frac{n-1}{2}~\text{times}},\
\underbrace{
\frac{-1-\sqrt{n}}{2},\dots,\frac{-1-\sqrt{n}}{2}
}_{\frac{n-1}{2}~\text{times}}
\Bigg\}.
\]
By Lemma~\ref{lem:t}, its Seidel spectrum is therefore
\[
\operatorname{Spec}(S(G)) =
\Big\{
0,\ 
\underbrace{-\sqrt{n},\dots,-\sqrt{n}}_{\frac{n-1}{2}~\text{times}},\
\underbrace{\sqrt{n},\dots,\sqrt{n}}_{\frac{n-1}{2}~\text{times}}
\Big\}.
\]
Consequently,
\[
\begin{aligned}
\mathcal{E}_p(S(G))
&=
\frac{n-1}{2}|\sqrt n|^p
+
\frac{n-1}{2}|-\sqrt n|^p\\
&=
(n-1)n^{\frac p2}.
\end{aligned}
\]
Hence equality holds if and only if $G$ is a conference graph.

Finally, when $n$ is a prime power with $n\equiv1\pmod4$, the Paley graph $P(n)$ is a conference graph, and therefore it attains equality.
\end{proof}

\begin{remark}
The lower bound in Theorem~\ref{conference} is attained only for those orders $n$ for which a conference graph of order $n$ exists. Indeed,
if no such conference graph exists, then the equality characterization in Theorem~\ref{conference} implies that
\[
\mathcal{E}_p(S(G))
>
(n-1)n^{\frac{p}{2}}
\]
for every $\frac{n-1}{2}$-regular graph $G$ of order $n$.

Since there are only finitely many graphs of order $n$, the minimum $p$-Seidel energy over all $\frac{n-1}{2}$-regular graphs of order $n$
nevertheless exists; however, Theorem~\ref{conference} does not determine its exact value when no conference graph of order $n$ exists.

In particular, when $n$ is a prime power with
$n\equiv1\pmod4$, the Paley graph $P(n)$ is a conference graph. Hence, for such orders $n$, the lower bound is attained and
\[
\min \mathcal{E}_p(S(G))
=
(n-1)n^{\frac{p}{2}},
\]
where the minimum is taken over all $\frac{n-1}{2}$-regular graphs $G$ of order $n$.
\end{remark}

\section{Maximum $p$-Seidel energy of $r$-regular graphs, $p\geq3$}\label{secc4}
In mathematical optimization, the method of Lagrange multipliers is a strategy for finding the local extrema (i.e., maxima and minima) of a function subject to constraint equations (i.e., subject to the condition that one or more equations have to be satisfied exactly by the chosen values of the variables) \cite{LDH1989}. The following lemma summarizes the method of Lagrange multipliers.

\begin{lemma}\label{lem:Lagrange}\rm\cite[Theorem, Page 327]{DGL2008}
Let $F:\mathbb{R}^n\to\mathbb{R}$ and $g_1,\dots,g_k:\mathbb{R}^n\to\mathbb{R}$ be continuously differentiable functions, and consider the constrained optimization problem
\[
\text{maximize/minimize } F(x)
\quad\text{subject to}\quad
g_1(x)=0,\ \dots,\ g_k(x)=0.
\]
Assume that $x^\ast\in\mathbb{R}^n$ is a local extremum and that the gradients $\nabla g_1(x^\ast),\dots,\nabla g_k(x^\ast)$ are linearly independent. Then there exist real numbers $\alpha_1,\dots,\alpha_k$ such that
\begin{equation}\label{eq:Lagrange-system}
\nabla F(x^\ast)
=\alpha_1\,\nabla g_1(x^\ast)+\cdots+\alpha_k\,\nabla g_k(x^\ast),
\end{equation}
together with the constraints
\[
g_1(x^\ast)=\cdots=g_k(x^\ast)=0.
\]
Thus, any interior constrained extremum must satisfy the system \eqref{eq:Lagrange-system}.
\end{lemma}

\begin{lemma}\label{lem:second-order}\rm\cite[Section 11.5, p.~333]{DGL2008}
Let $F:\mathbb{R}^n\to\mathbb{R}$ and
$g_1,\dots,g_k:\mathbb{R}^n\to\mathbb{R}$ be twice continuously differentiable functions. Suppose that $x^\ast\in\mathbb{R}^n$ is a local maximum of $F$ subject to
\(g_1(x)=0,\dots,g_k(x)=0\), and that the gradients
\[
\nabla g_1(x^\ast),\dots,\nabla g_k(x^\ast)
\]
are linearly independent. Let $\alpha_1,\dots,\alpha_k$ be the Lagrange multipliers satisfying
\[
\nabla F(x^\ast)
=
\alpha_1\nabla g_1(x^\ast)
+\cdots+
\alpha_k\nabla g_k(x^\ast).
\]
Then, for every $u\in\mathbb{R}^n$ satisfying \(\nabla g_j(x^\ast)^{\top}u=0\), \(j=1,\dots,k\), we have
\[
u^{\top}
\left(
\nabla^2F(x^\ast)
-
\sum_{j=1}^k
\alpha_j\nabla^2g_j(x^\ast)
\right)u
\le0.
\]
\end{lemma}
To determine the maximum $p$-Seidel energy among all $r$-regular graphs, we first solve the following constrained optimization problem. Thus, the optimization problem with constraints is stated in the following lemma, which is valid for $p\geq3$.

\begin{lemma}\label{L333}
Let $n\ge6$ and $p\geq3$. Consider the following problem
\[
F=F(t_1,\dots,t_{n-1})=\sum_{i=1}^{n-1}|t_i|^p
\]
subject to the constraints
\[
\sum_{i=1}^{n-1} t_i = 1,~~~~ \sum_{i=1}^{n-1} t_i^2 = n(n-1)-1,~~~~~ |t_i|\le n-1\ \forall i=1,\ldots,n-1.
\]
Then the maximum of $F$ is attained by all vectors that are $\big\{n-1,~\underbrace{-1,-1,\dots,-1}_{(n-2)~\text{times}}\big\}$ up to permutation.
\end{lemma}

\begin{proof}Let us consider the feasible set to be 
\[
\mathcal{M}
=\Bigl\{
(t_1,\dots,t_{\,n-1})\in\mathbb{R}^{\,n-1} :
\ \sum_{i=1}^{n-1} t_i = 1,\ 
\sum_{i=1}^{n-1} t_i^{\,2} = n(n-1)-1,\ 
|t_i|\le n-1\ \forall i
\Bigr\}.
\]
Then $\mathcal{M}$ is non-empty as $t^{(0)}=\{n-1,~\underbrace{-1,-1,\dots,-1}_{(n-2)~\text{times}}\}\in\mathcal{M}$. The set $\mathcal{M}$ is closed and bounded, hence compact. Since $F$ is continuous, the maximum of $F$ over $\mathcal{M}$ is attained at some
\[
t^\ast=(t_1^\ast,\dots,t_{n-1}^\ast)\in\mathcal{M}.
\]
\textbf{Case (1):} Suppose that $t^\ast$ has a coordinate on the boundary, i.e., $|t_j^\ast|=n-1~~\text{for some } j$. Without loss of generality, assume that $j=z$, $1\le z\le(n-1)$.\\\\
\textbf{Subcase (1.1):} Let $t_z^\ast = n-1$.

We have
\[
\sum_{i\ne z,~i=1}^{n-1} t_i^\ast
= 1 - t_z^\ast
= 1 - (n-1)
= 2-n.
\]
Thus the arithmetic mean of $t_1^\ast,t_2^\ast, \dots,t_{z-1}^\ast,t_{z+1}^\ast,\dots,t_{n-1}^\ast$ is
\[
\frac{1}{n-2}\sum_{i\ne z,~i=1}^{n-1} t_i^\ast
= \frac{2-n}{n-2}
= -1.
\]
We have
\[
\sum_{i\ne z,~i=1}^{n-1} (t_i^\ast)^2
= \bigl(n(n-1)-1\bigr) - (t_z^\ast)^2
= n(n-1)-1-(n-1)^2
= n-2.
\]
Thus the arithmetic mean of the squares is
\[
\frac{1}{n-2}\sum_{i\ne z,~i=1}^{n-1} (t_i^\ast)^2
= \frac{n-2}{n-2}
=1.
\]
By the Cauchy-Schwarz inequality,
\[
\left(\sum_{i\ne z}t_i^\ast\right)^2
\le
(n-2)\sum_{i\ne z}(t_i^\ast)^2.
\]
Since
\[
\left(\sum_{i\ne z}t_i^\ast\right)^2
=(n-2)^2
=(n-2)\sum_{i\ne z}(t_i^\ast)^2.
\]
Hence equality holds in the Cauchy-Schwarz inequality. Thus all $t_i^\ast$, $i\ne z$, are equal. Since
\[
\sum_{i\ne z}t_i^\ast=2-n,
\]
we obtain \(t_i^\ast=-1\), \(i\ne z\). Therefore,
\[
t^\ast\ \text{is}\ \left\{
n-1,
\underbrace{-1,\ldots,-1}_{n-2\text{ times}}
\right\} ~\text{up to permutation}.
\]
\textbf{Subcase (1.2):} Let $t_z^\ast = -(n-1)$.

We have
\[
\sum_{i\ne z,~i=1}^{n-1} t_i^\ast
= 1 - t_z^\ast
= 1 + (n-1)
= n.
\]
Thus the arithmetic mean of $t_1^\ast,t_2^\ast, \dots,t_{z-1}^\ast,t_{z+1}^\ast,\dots,t_{n-1}^\ast$ is
\[
\frac{1}{n-2} \sum_{i\ne z,~i=1}^{n-1} t_i^\ast
= \frac{n}{n-2}.
\]
Again we have
\[
\sum_{i\ne z,~i=1}^{n-1} (t_i^\ast)^2
= \bigl(n(n-1)-1\bigr) - (t_z^\ast)^2
= n(n-1)-1-(n-1)^2
= n-2.
\]
Thus, the arithmetic mean of the squares is
\[
\frac{1}{n-2}\sum_{i\ne z,~i=1}^{n-1} (t_i^\ast)^2
= \frac{n-2}{n-2}
=1.
\]
By the Cauchy-Schwarz inequality,
\[
\left(\sum_{i\ne z}t_i^\ast\right)^2
\le
(n-2)\sum_{i\ne z}(t_i^\ast)^2.
\]
Hence
\[
n^2\le(n-2)^2,
\]
which is impossible for $n\ge3$. Hence, this subcase is impossible for a feasible point (a point that satisfies all the given constraints).

Combining the two subcases, we conclude that 
\[
t^\ast\ \text{is}\ \left\{
n-1,
\underbrace{-1,\ldots,-1}_{n-2\text{ times}}
\right\} ~\text{up to permutation}.
\]
\textbf{Case (2):} Assume that \(|t_i^\ast|<n-1\), \(i=1,\ldots,n-1\). Thus $t^\ast$ lies in the interior with respect to the constraints $|t_i|\le n-1$, while still satisfying the two equality constraints. Therefore, locally, we only need to consider the two equality constraints.

Let us define
\[
g_1(t)=g_1(t_1,\ldots,t_{n-1})=\sum_{i=1}^{n-1} t_i - 1,\quad
g_2(t)=g_2(t_1,\ldots,t_{n-1})=\sum_{i=1}^{n-1} t_i^2 - \bigl(n(n-1)-1\bigr).
\]
The gradients of the constraints are

\[\nabla g_1(t)=\big(\underbrace{1,1,\dots,1}_{(n-1)~\text{times}}\big),\qquad
\nabla g_2(t)=\big(\underbrace{2t_1,2t_2,\dots,2t_{n-1}}_{(n-1)~\text{times}}\big).
\]
The vectors $\nabla g_1(t^\ast)$ and $\nabla g_2(t^\ast)$ are linearly independent. Indeed, if they were linearly dependent, then $t_1^\ast=\cdots=t_{n-1}^\ast$. The first constraint would then give $t_i^\ast=\frac{1}{n-1}$ for every $i$, which contradicts the second constraint. Therefore, the hypotheses of Lemma~\ref{lem:Lagrange} are satisfied.

Therefore, by Lemma \ref{lem:Lagrange} at $t^\ast$ there exist $\alpha_1,~\alpha_2\in\mathbb{R}$ such that
\begin{equation}\label{eq:Lagrange-vector}
\nabla F(t^\ast)=\alpha_1\,\nabla g_1(t^\ast)+\alpha_2\,\nabla g_2(t^\ast),
\end{equation}
together with the conditions
\[
g_1(t^\ast)=0, \quad g_2(t^\ast)=0.
\]
Now computing $\nabla F$ in each coordinate $t_j$ we have
\[
\frac{\partial}{\partial t_j}|t_j|^p
=p\,|t_j|^{p-1}\operatorname{sgn}(t_j).
\]
Thus
\[
\Big(\nabla F(t)\Big)_j = p\,t_j|t_j|^{p-2}.
\]
Computing $\nabla g_1,~\nabla g_2$ in each coordinate $t_j$ we have 
\[
\frac{\partial g_1}{\partial t_j}=1,~~\frac{\partial g_2}{\partial t_j}=2t_j.
\]
Substituting into equality (\ref{eq:Lagrange-vector}), from the $j$-th coordinate of that vector equation, we obtain
\[
p\,|t^\ast_j|^{p-1}\operatorname{sgn}(t^\ast_j)
=\alpha_1\cdot 1 + \alpha_2\cdot 2 t_j^\ast= \alpha_1 + 2\alpha_2 t_j^\ast.
\]
Finally,
\begin{equation}\label{eq:Lag-coordinate}
p\,|t^\ast_j|^{p-1}\operatorname{sgn}(t^\ast_j)-\alpha_1-2\alpha_2\,t_j^\ast=0,
~~~j=1,\dots,n-1.
\end{equation}
We now prove that an interior maximizer can have at most two distinct coordinate values. 

Define \(f(x)=|x|^p\) and \(h(x)=f'(x)-\alpha_1-2\alpha_2x=p\,x|x|^{p-2}-\alpha_1-2\alpha_2x\).

By \eqref{eq:Lag-coordinate}, every coordinate $t_i^\ast$ is a zero of $h$. Moreover,
\[
h'(x)
=
p(p-1)|x|^{p-2}-2\alpha_2.
\]
If $\alpha_2\le0$, then $h$ is strictly increasing on $\mathbb R$. Hence, $h$ has at most one zero, which would imply that all coordinates $t_i^\ast$ are equal. This is impossible because of the two constraints. Therefore, \(\alpha_2>0\).

Set \(\rho=\left(\frac{2\alpha_2}{p(p-1)}
\right)^{\frac{1}{p-2}}\). Then \(h'(x)>0 \quad \text{for } |x|>\rho,
\qquad
h'(x)<0 \quad \text{for } |x|<\rho\).

Thus $h$ is increasing on $(-\infty,-\rho)$, decreasing on $(-\rho,\rho)$, and increasing on $(\rho,\infty)$. Hence $h$ has at most three distinct real zeros.

Let $u=(u_1,\ldots,u_{n-1})$ satisfy
\[
\sum_{i=1}^{n-1}u_i=0
\qquad\text{and}\qquad
\sum_{i=1}^{n-1}t_i^\ast u_i=0.
\]
Since \(\nabla g_1(t^\ast)=(1,\ldots,1)\), \(\nabla g_2(t^\ast)=2(t_1^\ast,\ldots,t_{n-1}^\ast)\), these conditions are equivalent to
\[
\nabla g_1(t^\ast)^\top u=0,
\qquad
\nabla g_2(t^\ast)^\top u=0.
\]
Moreover, \(\nabla^2g_1=0\), \(\nabla^2g_2=2I\), and
\[
\nabla^2F(t^\ast)
=
\operatorname{diag}\left(
p(p-1)|t_1^\ast|^{p-2},\ldots,
p(p-1)|t_{n-1}^\ast|^{p-2}
\right).
\]
Therefore, by Lemma~\ref{lem:second-order},
\begin{equation}\label{eq:second-order}
\sum_{i=1}^{n-1}
\left(
p(p-1)|t_i^\ast|^{p-2}-2\alpha_2
\right)u_i^2
\le0.
\end{equation}
Suppose that a value $x$ occurs at least twice among the coordinates of $t^\ast$. Choose two coordinates equal to $x$, and let $u$ have entries $1$ and $-1$ at these two coordinates, respectively, and $0$ elsewhere. Then
\[
\sum_{i=1}^{n-1}u_i=0
\qquad\text{and}\qquad
\sum_{i=1}^{n-1}t_i^\ast u_i=0.
\]
Hence, by \eqref{eq:second-order},
\[
2\left(
p(p-1)|x|^{p-2}-2\alpha_2
\right)\le0.
\]
Therefore,
\[
h'(x)=p(p-1)|x|^{p-2}-2\alpha_2\le0,
\]
and consequently
\[
|x|\le\rho.
\]
Since $h$ is strictly decreasing on $(-\rho,\rho)$, at most one distinct zero of $h$ lying in $[-\rho,\rho]$ can occur more than once among the coordinates of $t^\ast$.

We show by contradiction that three distinct coordinate values, say \(a<b<c\), of $t^\ast$ are impossible. 
Since every
coordinate value is a zero of $h$, and $h$ is strictly increasing on $(-\infty,-\rho)$, strictly decreasing on $(-\rho,\rho)$, and strictly increasing on $(\rho,\infty)$, the existence of three distinct zeros implies
\[
a<-\rho<b<\rho<c.
\]
Therefore,
\[
h'(a)>0,\qquad h'(b)<0,\qquad h'(c)>0.
\]
Since every coordinate value that occurs at least twice must satisfy
$h'(x)\le0$, neither $a$ nor $c$ can occur more than once.
Consequently, $a$ and $c$ occur exactly once, while $b$ occurs
$n-3$ times.

Let \(d_a=h'(a)\), \(d_b=h'(b)\), \(d_c=h'(c)\). Since $a<b<c$ are three distinct zeros of $h$, and $h$ is strictly increasing on $(-\infty,-\rho)$, strictly decreasing on $(-\rho,\rho)$, and strictly increasing on $(\rho,\infty)$, we have
\[
a<-\rho<b<\rho<c.
\]
Therefore,
\[
d_a>0,\qquad d_b<0,\qquad d_c>0.
\]
Moreover, since $h(a)=h(b)=h(c)=0$, we have
\[
f'(b)-f'(a)=2\alpha_2(b-a)
\]
and
\[
f'(c)-f'(b)=2\alpha_2(c-b).
\]
Hence
\[
\frac{f'(b)-f'(a)}{b-a}
=
2\alpha_2
=
\frac{f'(c)-f'(b)}{c-b}.
\]
Since \(f''(x)=p(p-1)|x|^{p-2}\) is convex for $p\ge3$, for every $s\in[0,1]$ we have
\[
f''((1-s)a+sb)
\le
(1-s)f''(a)+s f''(b).
\]
Integrating with respect to $s$ over $[0,1]$, we obtain
\[
\int_0^1 f''((1-s)a+sb)\,ds
\le
\frac{f''(a)+f''(b)}{2}.
\]
Using the substitution \(x=(1-s)a+sb\), we get
\[
\int_0^1 f''((1-s)a+sb)\,ds
=
\frac{1}{b-a}\int_a^b f''(x)\,dx.
\]
Therefore,
\[
\frac{1}{b-a}\int_a^b f''(x)\,dx
\le
\frac{f''(a)+f''(b)}{2}.
\]
By the fundamental theorem of calculus,
\[
\frac{1}{b-a}\int_a^b f''(x)\,dx
=
\frac{f'(b)-f'(a)}{b-a}
=
2\alpha_2.
\]
Hence
\[
2\alpha_2
\le
\frac{f''(a)+f''(b)}{2}.
\]
Similarly,
\[
2\alpha_2
\le
\frac{f''(b)+f''(c)}{2}.
\]
Since \(d_a=f''(a)-2\alpha_2\), \(d_b=f''(b)-2\alpha_2\), \(d_c=f''(c)-2\alpha_2\).
Therefore, \(d_a+d_b\ge0\), \(d_b+d_c\ge0\), and consequently, \(d_a\ge-d_b\), \(d_c\ge-d_b\).

Let \(m=n-3\ge3\). Since the multiplicities of $a,b,c$ are $1,m,1$, respectively, we choose
\[
u=
\left(
c-b,\,
\underbrace{-\frac{c-a}{m},\ldots,-\frac{c-a}{m}}_{m\text{ times}},\,
b-a
\right).
\]
Then
\[
\sum_{i=1}^{n-1}u_i
=
(c-b)-(c-a)+(b-a)=0
\]
and
\[
\sum_{i=1}^{n-1}t_i^\ast u_i
=
a(c-b)-b(c-a)+c(b-a)=0.
\]
Hence $u$ is an admissible vector in
\eqref{eq:second-order}. Therefore,
\[
d_a(c-b)^2
+
\frac{d_b}{m}(c-a)^2
+
d_c(b-a)^2
\le0.
\]
On the other hand, using \(d_a\ge-d_b\), \(d_c\ge-d_b\), we obtain
\[
\begin{aligned}
&d_a(c-b)^2
+\frac{d_b}{m}(c-a)^2
+d_c(b-a)^2\ge
-d_b\bigl((c-b)^2+(b-a)^2\bigr)
+\frac{d_b}{m}(c-a)^2.
\end{aligned}
\]
Since $d_b<0$ and $m\ge3$,
\[
\frac{d_b}{m}>\frac{d_b}{2}.
\]
Hence
\[
\begin{aligned}
&d_a(c-b)^2
+\frac{d_b}{m}(c-a)^2
+d_c(b-a)^2>
-d_b\bigl((c-b)^2+(b-a)^2\bigr)
+\frac{d_b}{2}(c-a)^2\\
&\hspace{5.5cm}=
-\frac{d_b}{2}
\bigl((c-b)-(b-a)\bigr)^2
\ge0,
\end{aligned}
\]
which contradicts \eqref{eq:second-order}.

Therefore, an interior maximizer cannot have three distinct coordinate values. Hence, $t^\ast$ has at most two distinct coordinate values.

Let $a$ and $b$ be the two distinct coordinate values of $t^\ast$. We claim that one of the two values $a$ and $b$ must occur exactly once. Suppose, to the contrary, that both $a$ and $b$ occur at least twice.

Choose two coordinates equal to $a$. Let $u$ have entries $1$ and $-1$ at these two coordinates, respectively, and $0$ elsewhere. Then
\[
\sum_{i=1}^{n-1}u_i=0
\]
and
\[
\sum_{i=1}^{n-1}t_i^\ast u_i=a-a=0.
\]
Hence $u$ is admissible in \eqref{eq:second-order}. Therefore, \(2h'(a)\le0\), and thus \(h'(a)\le0\). Applying the same argument to two coordinates equal to $b$, we obtain \(h'(b)\le0\).

By \eqref{eq:Lag-coordinate}, every coordinate value of $t^\ast$ is a zero of $h$. Hence \(h(a)=h(b)=0\). Also, \(h'(x)\le0\) only for $|x|\le\rho$, and $h$ is strictly decreasing on $[-\rho,\rho]$. Therefore, $h$ can have at most one zero in this interval. But $a\ne b$, and both $a$ and $b$ are zeros of $h$ satisfying $h'(a)\le0$ and $h'(b)\le0$, which is impossible.

Therefore, one of the two values $a$ and $b$ must occur exactly once. 
Up to reordering, we may write
\[
t_1^\ast=a,\qquad
t_2^\ast=\cdots=t_{n-1}^\ast=b.
\]
The constraints give
\begin{align}
a+(n-2)b&=1, \label{eq:ab1}\\
a^2+(n-2)b^2&=n(n-1)-1. \label{eq:ab2}
\end{align}
From \eqref{eq:ab1},
\[
a=1-(n-2)b.
\]
Substituting this into \eqref{eq:ab2}, we obtain
\[
(1-(n-2)b)^2+(n-2)b^2=n(n-1)-1.
\]
Equivalently,
\[
(n-2)(b+1)\bigl((n-1)b-(n+1)\bigr)=0.
\]
Therefore, \(b=-1\) or \(b=\frac{n+1}{n-1}\).

If $b=-1$, then \(a=1-(n-2)(-1)=n-1\). Thus
\[
t^\ast
=
\left\{
n-1,
\underbrace{-1,\ldots,-1}_{n-2\text{ times}}
\right\},
\]
which is precisely the boundary solution obtained in Case~(1).

Now suppose that \(b=\frac{n+1}{n-1}\). Then
\[
a=1-(n-2)b
=-(n-1)+\frac{2}{n-1}.
\]
Let \(M=n-1\), \(A=M-\frac{2}{M}\), \(B=1+\frac{2}{M}\). Then \(|a|=A\), \(b=B\), and, since $n\ge6$,
\[
M>A>B>1.
\]
Since both the vector \((M,-1,\ldots,-1)\) and \((a,b,\ldots,b)\) satisfy the same sum-of-squares constraint, we have
\begin{equation}\label{eq:second-moment-relation}
M^2-A^2=(M-1)(B^2-1).
\end{equation}
Let \(\varphi(x)=x^{\frac{p}{2}}\). Since $p\ge3$, the derivative \(\varphi'(x)=\frac{p}{2}x^{\frac p2-1}\) is strictly increasing on $(0,\infty)$.

By the Mean Value Theorem, there exist \(\xi_1\in(A^2,M^2)\) and \(\xi_2\in(1,B^2)\) such that
\[
\frac{M^p-A^p}{M^2-A^2}
=\varphi'(\xi_1)\quad\text{and}\quad\frac{B^p-1}{B^2-1}=\varphi'(\xi_2).
\]
Since \(\xi_1>A^2>B^2>\xi_2\), we have
\[
\varphi'(\xi_1)>\varphi'(\xi_2).
\]
Therefore,
\[
\frac{M^p-A^p}{M^2-A^2}
>
\frac{B^p-1}{B^2-1}.
\]
Using \eqref{eq:second-moment-relation}, we obtain
\[
M^p-A^p>(M-1)(B^p-1).
\]
Hence
\[
A^p+(M-1)B^p<M^p+(M-1).
\]
Therefore,
\[
|a|^p+(n-2)|b|^p
<
(n-1)^p+(n-2).
\]
Thus this second solution cannot be a global maximizer. Consequently, the only global maximizers are
\[
t^\ast\ \text{is}\ \left\{
n-1,
\underbrace{-1,\ldots,-1}_{n-2\text{ times}}
\right\} ~\text{up to permutation}.
\]
This completes the proof.
\end{proof}
In the following result, Berman, Shaked-Monderer, Singh, and Zhang determine the Seidel spectrum of the complete bipartite graph $K_{a,b}$.

\begin{lemma}\rm\cite[Lemma 2.1(c)]{AB2019}\label{L45}
Let $K_{a,b}$ be the complete bipartite graph with $n=a+b$ vertices. Then the Seidel spectrum of $K_{a,b}$ is 
\[
\left\{
n-1,
\underbrace{-1,\ldots,-1}_{n-1\text{ times}}
\right\}.
\]
\end{lemma}

\begin{remark}
Lemma~\ref{L45} is consistent with the equality statement in Theorem~\ref{Th.2.1}. Indeed, for every $a,b\geq1$ with $a+b=n$,
\[
\operatorname{Spec}(S(K_{a,b}))
=
\left\{
n-1,
\underbrace{-1,\ldots,-1}_{n-1\text{ times}}
\right\}
\]
and therefore
\[
\mathcal{E}(S(K_{a,b}))=2(n-1)
=\mathcal{E}(S(K_n)).
\]
Although $K_{a,b}$ is generally not isomorphic to $K_n$, it is $SC$-equivalent to $K_n$. Thus, the equality in the minimum Seidel-energy bound is an $SC$-equivalence characterization, rather than an
isomorphism characterization.
\end{remark}

\begin{lemma}\rm\cite[Theorem 2.5.6]{RA2013}\label{234}
Every real symmetric matrix $A$ admits the spectral decomposition
\[
A \;=\; \sum_{i=1}^{t} \lambda_i\, u_i u_i^{\top},
\]
where $\lambda_1,\ldots,\lambda_t$ are the nonzero eigenvalues of $A$, and $u_1,\ldots,u_t$ are the corresponding orthonormal eigenvectors.
\end{lemma}
We now characterize the graphs whose Seidel spectrum is equal to Seidel spectrum of $K_{a,b}$, see  Lemma~\ref{L45}, up to isomorphism.

\begin{lemma}\label{lem:Seidel-complete-bipartite}
Let $G$ be a graph on $n\ge 2$ vertices. Then the following statements are equivalent:
\begin{enumerate}
\item \(\operatorname{Spec}(S(G))
=
\left\{
n-1,\,
\underbrace{-1,-1,\ldots,-1}_{n-1\text{ times}}
\right\}\).
\item There exists an integer $s$ with $1\le s\le n-1$, such that \(G\cong K_{s,n-s}\).
\end{enumerate}
\end{lemma}

\begin{proof}
$(1)\Rightarrow(2):$
Suppose that
\[
\operatorname{Spec}(S(G))
=
\left\{
n-1,\,
\underbrace{-1,-1,\ldots,-1}_{n-1\text{ times}}
\right\}.
\]
Since $S(G)$ is real symmetric, by Lemma~\ref{234} it admits an
orthonormal spectral decomposition. Let $u_1$ be a unit eigenvector
corresponding to the eigenvalue $n-1$, and let
$u_2,\ldots,u_n$ be orthonormal eigenvectors corresponding to the
eigenvalue $-1$. Then
\[
S(G)
=
(n-1)u_1u_1^{\top}
-
\sum_{i=2}^{n}u_i u_i^{\top}.
\]
Since \(\sum_{i=1}^{n}u_i u_i^{\top}=I\), we obtain
\begin{align*}
S(G)
&=
(n-1)u_1u_1^{\top}
-\bigl(I-u_1u_1^{\top}\bigr)\\
&=
-I+n u_1u_1^{\top}.
\end{align*}
Put $v=u_1$. Hence
\[
S(G)=-I+nvv^{\top}.
\]
Since every diagonal entry of a Seidel matrix is zero, \[
0=S(G)_{ii}=-1+n v_i^2,
\]
and therefore
\[
v_i^2=\frac1n
\qquad\text{for every }i=1,\ldots,n.
\]
Thus
\[
v_i=\pm\frac1{\sqrt n}.
\]
Define \(z_i=\sqrt n\,v_i\in\{-1,1\}\). For $i\neq j$,
\[
S(G)_{ij}
=
n v_i v_j
=
z_i z_j.
\]
Let \(V_+=\{\,v_i:z_i=1\,\}\), \(V_-=\{\,v_i:z_i=-1\,\}\). If two vertices $v_i,v_j$ belong to the same set, then $z_i z_j=1$, and consequently
\[
S(G)_{ij}=1.
\]
By the definition of the Seidel matrix, $v_i$ and $v_j$ are therefore non-adjacent. On the other hand, if $v_i\in V_+$ and $v_j\in V_-$, then $z_i z_j=-1$, and hence
\[
S(G)_{ij}=-1.
\]
Thus $v_i$ and $v_j$ are adjacent.

Consequently, there are no edges inside $V_+$ or $V_-$, while every vertex of $V_+$ is adjacent to every vertex of $V_-$. Therefore, $G$ is a complete bipartite graph.

Since $G$ has at least one edge, both $V_+$ and $V_-$ are nonempty. Let \(|V_+|=s\), \(|V_-|=n-s\) where $1\le s\le n-1$. Then we obtain
\[
G\cong K_{s,n-s}.
\]
$(2)\Rightarrow(1):$
Suppose that
\[
G\cong K_{s,n-s}
\]
for some $1\le s\le n-1$. By Lemma~\ref{L45}, the Seidel spectrum of every complete bipartite graph on $n$ vertices is
\[
\operatorname{Spec}(S(G))
=
\left\{
n-1,\,
\underbrace{-1,-1,\ldots,-1}_{n-1\text{ times}}
\right\}.
\]
This completes the proof.
\end{proof}

\begin{corollary}\label{L444}
Let $G$ be an $r$-regular graph on $n=2r$ vertices. Then the following statements are equivalent:
\begin{enumerate}
\item $\,\operatorname{Spec}(S(G))=\{n-1,~\underbrace{-1,-1,\dots,-1}_{(n-1)\text{ times}}\}$.
\item $\,G\cong K_{r,r}$.
\end{enumerate}
\end{corollary}

\begin{proof}
$(1)\Rightarrow(2):$
Suppose that
\[
\operatorname{Spec}(S(G))
=
\left\{
n-1,\,
\underbrace{-1,-1,\ldots,-1}_{(n-1)\text{ times}}
\right\}.
\]
By Lemma~\ref{lem:Seidel-complete-bipartite}, there exists an integer $s$, with $1\le s\le n-1$, such that
\[
G\cong K_{s,n-s}.
\]
In the complete bipartite graph $K_{s,n-s}$, every vertex in the part of size $s$ has degree $n-s$, while every vertex in the part of size $n-s$ has degree $s$.

Since $G$ is $r$-regular, all vertices of $G$ have the same degree $r$. Therefore, \(s=r\), \(n-s=r\). In particular, \(s=n-s\). Hence \(n=2s\). Since $n=2r$, it follows that \(s=r\). Therefore,
\[
G\cong K_{s,n-s}=K_{r,r}.
\]
$(2)\Rightarrow(1):$
Conversely, suppose that
\[
G\cong K_{r,r}.
\]
Since $n=2r$, the graph $K_{r,r}$ is a complete bipartite graph on $n$ vertices. By Lemma~\ref{L45}, the Seidel spectrum of every complete bipartite graph $K_{s,n-s}$ is
\[
\left\{
n-1,\,
\underbrace{-1,-1,\ldots,-1}_{(n-1)\text{ times}}
\right\}.
\]
Therefore,
\[
\operatorname{Spec}(S(G))
=
\left\{
n-1,\,
\underbrace{-1,-1,\ldots,-1}_{(n-1)\text{ times}}
\right\}.
\]
Hence $(1)$ and $(2)$ are equivalent.
\end{proof}
In the following theorem, we prove that among all $r$-regular graphs, the balanced complete bipartite graph (up to isomorphism) achieves the maximum $p$-Seidel energy with fixed order $n=2r$, where $p\geq3$.

\begin{theorem}\label{max}
Let $G$ be an $r$-regular graph on $n=2r$ vertices. Then for every $p\geq3$,
\[
\mathcal{E}_p(S(G)) \le (n-1)^p + (n-1),
\]
with equality if and only if $G\cong K_{r,r}$.
\end{theorem}
\begin{proof}
We first observe that every $r$-regular graph of order $n=2r$ is connected. Indeed, every component of an $r$-regular simple graph contains at least $r+1$ vertices. If $G$ were disconnected, then it
would have at least two components, and hence \(n\geq 2(r+1)=2r+2\), which contradicts $n=2r$.

We first consider the cases $r=1$ and $r=2$. If $r=1$, then $n=2$, and therefore
\[
G\cong K_{1,1}.
\]
If $r=2$, then $n=4$. Since every component of a $2$-regular simple graph is a cycle of length at least $3$, a $2$-regular graph on four vertices must be
\[
G\cong C_4\cong K_{2,2}.
\]
Thus the result follows in both cases.

Hence, assume that $r\ge3$.

Since $n=2r$, we have $n\ge6$, so
Lemma~\ref{L333} applies. Label the Seidel eigenvalues of $G$ as \(\lambda_1,\lambda_2,\ldots,\lambda_n\) so that $\lambda_1$ is the eigenvalue corresponding to $\mathbf 1$.
Since $G$ is $r$-regular and $n=2r$, Lemma~\ref{lem:t} gives
\[
\lambda_1=n-1-2r=-1.
\]
Set \(t_i=\lambda_{i+1}\), \(i=1,\ldots,n-1\). Since $\operatorname{tr}(S(G))=0$, we obtain
\[
\sum_{i=1}^{n-1}t_i
=
-\lambda_1
=
1.
\]
Moreover, using \(\sum_{i=1}^n\lambda_i^2=n(n-1)\), we have
\[
\sum_{i=1}^{n-1}t_i^2
=
n(n-1)-1.
\]
Also, every row of $S(G)$ has absolute row sum $n-1$. Hence \(\|S(G)\|_{\infty}=\max_{1\leq i\leq n}\sum_{j=1}^{n}|S(G)_{ij}|=n-1\). Since every eigenvalue $\lambda$ of a matrix satisfies \(
|\lambda|\leq \|S(G)\|_{\infty}\), every Seidel eigenvalue satisfies
\[
|\lambda_i|\leq n-1,\qquad i=1,\ldots,n.
\]
Thus,
\[
|t_i|\leq n-1,\qquad i=1,\ldots,n-1.
\]
Therefore, every $r$-regular graph $G$ of order $n=2r$ determines a feasible point \((t_1,\ldots,t_{n-1})\) of the optimization problem in Lemma~\ref{L333}. Consequently,
\[
\begin{aligned}
\mathcal E_p(S(G))
&=|\lambda_1|^p+\sum_{i=1}^{n-1}|t_i|^p\\
&=1+\sum_{i=1}^{n-1}|t_i|^p\\
&\le
1+\max_{\mathcal M}
\sum_{i=1}^{n-1}|t_i|^p,
\end{aligned}
\]
where $\mathcal M$ is the feasible set in Lemma~\ref{L333}.

By Lemma~\ref{L333},
\[
\max_{\mathcal M}\sum_{i=1}^{n-1}|t_i|^p
=
(n-1)^p+(n-2).
\]
Hence
\[
\mathcal E_p(S(G))
\le
(n-1)^p+(n-1).
\]
Now suppose that equality holds, that is, 
\[\mathcal E_p(S(G))=(n-1)^p+(n-1).
\] 
Then the associated feasible vector \((t_1,\ldots,t_{n-1})\) must attain the maximum in Lemma~\ref{L333}. Hence, up to ordering,
\[
(t_1,\ldots,t_{n-1})
=
\left(
n-1,
\underbrace{-1,\ldots,-1}_{n-2\text{ times}}
\right).
\]
Since $\lambda_1=-1$, we obtain
\[
\operatorname{Spec}(S(G))
=
\left\{
n-1,\,
\underbrace{-1,-1,\ldots,-1}_{(n-1)\text{ times}}
\right\}.
\]
Therefore, by Corollary~\ref{L444},
\[
G\cong K_{r,r}.
\]
Conversely, suppose that \(G\cong K_{r,r}\).

Since $n=2r$, the graph $K_{r,r}$ is an $r$-regular graph of order $n$. By Corollary~\ref{L444},
\[
\operatorname{Spec}(S(G))
=
\left\{
n-1,\,
\underbrace{-1,-1,\ldots,-1}_{(n-1)\text{ times}}
\right\}.
\]
Hence
\[
\mathcal E_p(S(G))
=
|n-1|^p+(n-1)|-1|^p
=
(n-1)^p+(n-1).
\]
Thus equality holds.
\end{proof}

\begin{remark}
For $n=2r$, we have \(\mathcal{E}_p(S(K_n)) = \mathcal{E}_p(S(K_{r,r})) = (n-1)^p+(n-1)\).

Indeed,
\[
\operatorname{Spec}(S(K_n))
=
\left\{
-(n-1),\,
\underbrace{1,1,\ldots,1}_{(n-1)\text{ times}}
\right\},
\]
whereas
\[
\operatorname{Spec}(S(K_{r,r}))
=
\left\{
n-1,\,
\underbrace{-1,-1,\ldots,-1}_{(n-1)\text{ times}}
\right\}.
\]
Thus, the Seidel eigenvalues of $K_n$ and $K_{r,r}$ differ only by sign, and hence their absolute values are the same. Therefore, they
have the same $p$-Seidel energy.

This does not contradict Theorem~\ref{max}, since the theorem considers only $r$-regular graphs of order $n=2r$. The graph $K_n$ is $(n-1)$-regular and, for $n>2$, does not belong to this class.
Hence, within the class considered in Theorem~\ref{max}, equality is attained if and only if $G\cong K_{r,r}$.
\end{remark}

Before proving the next result, we first note a useful relation between the Seidel matrices of a graph and its complement. Since \(A(\overline{G})=J-I-A(G)\), we have
\[
\begin{aligned}
S(\overline{G})
&=J-I-2A(\overline{G})\\
&=J-I-2\bigl(J-I-A(G)\bigr)\\
&=-J+I+2A(G)\\
&=-S(G).
\end{aligned}
\]
Therefore, if \(\lambda_1,\lambda_2,\ldots,\lambda_n\) are the Seidel eigenvalues of $G$, then \(
-\lambda_1,-\lambda_2,\ldots,-\lambda_n\) are the Seidel eigenvalues of $\overline{G}$. Consequently, for every $p>0$,
\[
\mathcal E_p(S(\overline{G}))
=\sum_{i=1}^n |-\lambda_i|^p
=\sum_{i=1}^n |\lambda_i|^p
=\mathcal E_p(S(G)).
\]

\begin{theorem}\label{thm:seidel-n-2r+2}
Let $r\geq 1$, $n=2r+2$, and $p\geq 3$. Among all $r$-regular graphs $G$ of order $n$,
\[
\mathcal E_p(S(G))\leq (n-1)^p+(n-1).
\]
Moreover, equality holds if and only if \(
G\cong K_{r+1}\cup K_{r+1}\).
\end{theorem}
\begin{proof}
Let $G$ be an $r$-regular graph of order \(n=2r+2\), and put \(H=\overline{G}\). Since every vertex of $G$ has degree $r$, every vertex of $H$ has degree \(n-1-r=(2r+2)-1-r=r+1\). Thus, $H$ is an $(r+1)$-regular graph of order \(n=2(r+1)\).

Put $s=r+1$. Then $H$ is a $s$-regular graph of order $n=2s$. Therefore, by Theorem~\ref{max},
\[
\mathcal E_p(S(H))
\leq
(n-1)^p+(n-1),
\]
with equality if and only if \(H\cong K_{s,s}\).

Since \(S(H)=S(\overline{G})=-S(G)\), we have
\[
\mathcal E_p(S(H))
=
\mathcal E_p(S(G)).
\]
Hence
\[
\mathcal E_p(S(G))
\leq
(n-1)^p+(n-1).
\]
Now suppose that equality holds. Then equality also holds for $H$, and hence, by Theorem~\ref{max},
\[
H\cong K_{s,s}=K_{r+1,r+1}.
\]
Since $H=\overline{G}$, taking complements gives
\[
G
\cong
\overline{K_{r+1,r+1}}
=
K_{r+1}\cup K_{r+1}.
\]
Conversely, suppose that \(G\cong K_{r+1}\cup K_{r+1}\). Then
\[
H=\overline{G}\cong K_{r+1,r+1}=K_{s,s}.
\]
By the equality case of Theorem~\ref{max},
\[
\mathcal E_p(S(H))
=
(n-1)^p+(n-1).
\]
Since \(\mathcal E_p(S(G))
=
\mathcal E_p(S(H))\), we obtain
\[
\mathcal E_p(S(G))
=
(n-1)^p+(n-1).
\]
Therefore, equality holds if and only if
\[
G\cong K_{r+1}\cup K_{r+1}.
\]
\end{proof}

\section{Some open problems}\label{seccc5}
We conclude with the following open
questions.

\begin{question}
Let $p>2$. Characterize the graphs that maximize the $p$-Seidel energy among all graphs with fixed order $n$.
\end{question}

\begin{question}
Let $0<p<2$. Characterize the graphs that minimize the $p$-Seidel energy among all graphs with fixed order $n$.
\end{question}

\begin{question}
Let $0<p<2$. Characterize the $r$-regular graphs that maximize (or minimize) the $p$-Seidel energy among all $r$-regular graphs with fixed order $n$.
\end{question}

\begin{question}
Let $p\geq3$. For the parameter pairs $(n,r)$ not covered by Theorems~\ref{max} and~\ref{thm:seidel-n-2r+2}, characterize the $r$-regular graphs that maximize the $p$-Seidel energy among all $r$-regular graphs of order $n$.
\end{question}

\section*{Acknowledgements} 
The authors express sincere gratitude to the Department of Mathematics at Bar-Ilan University (BIU). The second author also gratefully acknowledges the postdoctoral financial support provided by Bar-Ilan University (BIU), Israel.

\end{document}